\documentclass[11pt]{article}%
\usepackage{amsmath}
\usepackage{amsfonts}
\usepackage{amssymb}
\usepackage{graphicx}%
\newtheorem{theorem}{Theorem}

\newtheorem{corollary}[theorem]{Corollary}

\newtheorem{lemma}[theorem]{Lemma}

\newtheorem{remark}[theorem]{Remark}

\newenvironment{proof}[1][Proof]{\noindent\textbf{#1.} }{\ \rule{0.5em}{0.5em}}
\graphicspath{    {converted_graphics/}    {/}}
\begin{document}

\begin{center}
{\Large Periodic and Non-Periodic Solutions of a Ricker-type}

\smallskip

{\Large Second-Order Equation with Periodic Parameters}\footnote{Journal of Difference
Equations and Applications, DOI 10.1080/10236198.2016.1187142}

\medskip

N. LAZARYAN and H. SEDAGHAT \footnote{Department of Mathematics, Virginia
Commonwealth University, Richmond, VA \ 23284, USA;
\par
Email: hsedagha@vcu.edu}
\end{center}

\medskip

\begin{abstract}
We study the dynamics of the positive solutions of the exponential difference
equation
\[
x_{n+1}=x_{n-1}e^{a_{n}-x_{n-1}-x_{n}}%
\]

\noindent where the sequence $\{a_{n}\}$ is periodic. We find that
qualitatively different dynamics occurs depending on whether the period $p$ of
$\{a_{n}\}$ is odd or even. If $p$ is odd then periodic and non-periodic
solutions coexist (with different initial values) if the amplitudes of the
terms $a_{n}$ are allowed to vary over a sufficiently large range. But if $p$
is even then all solutions converge to an asymptotically stable limit cycle of
period $p$ if either all the odd-indexed or all the even-indexed terms of
$\{a_{n}\}$ are less than 2, and the sum of the even terms of $\{a_{n}\}$ does
not equal the sum of its odd terms. The key idea in this analysis is a
semiconjugate factorization of the above equation into a triangular system of
two first-order equations.

\end{abstract}

\medskip

\section{Introduction}

We study the behavior of solutions of the second-order difference equation%
\begin{equation}
x_{n+1}=x_{n-1}e^{a_{n}-x_{n-1}-x_{n}} \label{pce}%
\end{equation}
where the parameter $\{a_{n}\}$ is a periodic sequence of real numbers. This
equation is a special case of a stage-structured population model with a
Ricker-type recruitment function; see \cite{LS}. For more information and
additional Ricker-type models see, e.g. \cite{Liz}, \cite{LEO}, \cite{R} and
\cite{Z}.

Equation (\ref{pce}) has a rich variety of periodic and non-periodic
solutions. It exhibits coexisting periodic solutions if the range of
variation, or amplitude of $a_{n}$ is limited. We also show that an expanded
range or greater amplitudes for $a_{n}$ leads to the occurrence of coexisting
non-periodic solutions (including chaotic solutions) as well.

In addition, there is an unexpected qualitative difference between the
behavior of solutions when the period $p$ of $\{a_{n}\}$ is odd and when $p$
is even. When $p$ is odd different periodic and non-periodic stable solutions
are generated from different pairs of initial values; see Corollary \ref{op}
and Theorem \ref{oddgen} below. But if $p$ is even and the amplitude of
$a_{n}$ is less than 2 then Theorem \ref{peven} shows that asymptotically
stable $p$-cycles occur in a generic fashion, i.e. independently of initial values.

Such differences in behavior among the solutions of (\ref{pce}) are readily
explained by a semiconjugate factorization of (\ref{pce}) into a triangular
system of two first-order difference equations. The latter pair of equations
determine the full structure of (\ref{pce}) and allow us to explain the
aforementioned variety of behaviors that its solutions exhibit as follows: One
of the two first-order equations has periodic, hence bounded solutions when
$p$ is odd; see Lemma \ref{ptk}. However, if $p$ is even then the solutions of
the same first-order equation are unbounded (except for a boundary case); see
Lemma \ref{pet} and the remarks following it. Thus the corresponding
\textit{bounded} solutions of (\ref{pce}) \textquotedblleft forget" the
initial values and approach a single asymptotically stable solution.

The paper is divided into sections that highlight the natural dichotomy that
exists between odd and even periods. Detailed calculations that appear as a
sequence of lemmas are necessary to prove the aforementioned statements. A
careful review of the lemmas may be deferred until after Corollary \ref{op},
Theorem \ref{oddgen} and Theorem \ref{peven} have been examined along with the
remarks and figures that illustrate their statements.

\section{Order reduction}

It is convenient to note that if the initial values $x_{-1},x_{0}$ are
positive then $x_{n}>0$ for all $n\geq1$ and
\[
x_{n+1}<x_{n-1}e^{a_{n}-x_{n-1}}\leq e^{a_{n}}\frac{1}{e}=e^{a_{n}-1}%
\]

Thus the following result is obvious.

\begin{lemma}
\label{bd}Let $\{a_{n}\}$ be a sequence of real numbers that is bounded from
above with $\sup_{n}a_{n}=a$. If $x_{-1},x_{0}>0$ then the corresponding
solution $\{x_{n}\}$ of (\ref{pce}) is bounded and for all $n$%
\begin{equation}
0<x_{n}<e^{a-1}. \label{b}%
\end{equation}

\end{lemma}

The study of (\ref{pce}) is facilitated by the fact that it admits a
semiconjugate factorization that decomposes it into two equations of order
one. Following \cite{FSOR}, we define%
\[
t_{n}=\frac{x_{n}}{x_{n-1}e^{-x_{n-1}}}%
\]
for each $n\geq1$ and note that
\[
t_{n+1}t_{n}=\frac{x_{n+1}}{x_{n}e^{-x_{n}}}\frac{x_{n}}{x_{n-1}e^{-x_{n-1}}%
}=\frac{x_{n+1}}{x_{n-1}e^{-x_{n-1}-x_{n}}}=e^{a_{n}}%
\]
or equivalently,%
\begin{equation}
t_{n+1}=\frac{e^{a_{n}}}{t_{n}}. \label{sc1}%
\end{equation}

Now%
\begin{equation}
x_{n+1}=e^{a_{n}}x_{n-1}e^{-x_{n-1}}e^{-x_{n}}=e^{a_{n}}\frac{x_{n}}{t_{n}%
}e^{-x_{n}}=\frac{e^{a_{n}}}{t_{n}}x_{n}e^{-x_{n}}=t_{n+1}x_{n}e^{-x_{n}}
\label{sc2}%
\end{equation}

The pair of equations (\ref{sc1}) and (\ref{sc2}) constitute the semiconjugate
factorization of (\ref{pce}):%
\begin{align}
t_{n+1}  &  =\frac{e^{a_{n}}}{t_{n}},\quad t_{0}=\frac{x_{0}}{x_{-1}%
e^{-x_{-1}}}\label{star1}\\
x_{n+1}  &  =t_{n+1}x_{n}e^{-x_{n}} \label{star2}%
\end{align}

Every solution $\{x_{n}\}$ of (\ref{pce}) is generated by a solution of the
system (\ref{star1})-(\ref{star2}). Using the initial values $x_{-1},x_{0}$ we
obtain a solution $\{t_{n}\}$ of the first-order equation (\ref{star1}),
called the factor equation. This solution is then used to obtain a solution of
the cofactor equation (\ref{star2}) and thus also of (\ref{pce}). The system
(\ref{star1})-(\ref{star2}) is said to be triangular basically because one
equation (i.e. the factor equation) is independent of the other; see \cite{Sm}
for more information on triangular systems.

For an arbitrary sequence $\{a_{n}\}$ and a given $t_{0}\not =0$ by iterating
(\ref{star1}) we obtain%

\[
t_{1}=\frac{e^{a_{0}}}{t_{0}},\quad t_{2}=\frac{e^{a_{1}}}{t_{1}}%
=t_{0}e^{-a_{0}+a_{1}},\quad t_{3}=\frac{e^{a_{2}}}{t_{2}}=\frac{1}{t_{0}%
}e^{a_{0}-a_{1}+a_{2}},\quad t_{4}=\frac{e^{a_{3}}}{t_{3}}=t_{0}%
e^{-a_{0}+a_{1}-a_{2}+a_{3}},\cdots
\]

This pattern of development implies the following result.

\begin{lemma}
\label{bsc}Let $\{a_{n}\}$ be an arbitrary sequence of real numbers and
$t_{0}\not =0.$

(a) The general solution of (\ref{star1}) is given by
\begin{equation}
t_{n}=t_{0}^{(-1)^{n}}e^{(-1)^{n}s_{n}},\;\;n=1,2,... \label{gen1}%
\end{equation}

where%
\begin{equation}
s_{n}=\sum_{j=1}^{n}(-1)^{j}a_{j-1} \label{gen2}%
\end{equation}

(b) For all $n,$%
\[
x_{n}\leq\frac{1}{e}t_{n}.
\]

\end{lemma}

\begin{proof}
(a) For $n=1,$ (\ref{gen1}) yields%
\[
t_{1}=t_{0}^{-1}e^{-s_{1}}=\frac{1}{t_{0}}e^{-(-a_{0})}=\frac{e^{a_{0}}}%
{t_{0}}%
\]
which is true. Suppose that (\ref{gen1}) is true for $n\leq k.$ Then by
(\ref{gen1}) and (\ref{gen2})%
\[
t_{0}^{(-1)^{k+1}}e^{(-1)^{k+1}s_{k+1}}=\frac{1}{t_{0}^{(-1)^{k}}%
e^{(-1)^{k}s_{k}}}e^{(-1)^{2k+2}a_{k}}=\frac{e^{a_{k}}}{t_{k}}=t_{k+1}%
\]
which is again true and the proof is now complete by induction.

(b) This is an immediate consequence of (\ref{star2}) and the fact that
$xe^{-x}\leq1/e.$
\end{proof}

\medskip

In the sequel, whenever the sequence $\{a_{n}\}$ has period $p$ the following
quantity plays an essential role:%

\begin{equation}
\sigma=s_{p}=\sum_{j=1}^{p}(-1)^{j}a_{j-1}=-a_{0}+a_{1}-a_{2}+\ldots-a_{p-1}
\label{sig}%
\end{equation}

The following special-case result will be useful later on.

\begin{lemma}
\label{genper} Assume that $\{a_{n}\}$ is periodic with minimal period $p$. If
$\sigma=0$ and $t_{0}=1$, then $\{t_{n}\}$ is periodic with period $p$.
\end{lemma}

\begin{proof}
If $\sigma=0$, then by (\ref{gen1}) and (\ref{gen2}) in Lemma \ref{bsc} we
have:%
\[
t_{p}=t_{0}^{(-1)^{p}}e^{(-1)^{p}s_{p}}=e^{(-1)^{p}\sigma}=1=t_{0}%
\]

and%
\[
t_{n+p}=t_{0}^{(-1)^{n+p}}e^{(-1)^{n+p}s_{n+p}}=e^{(-1)^{n+p}s_{n+p}}.
\]

Now since $\sigma=0,$%
\[
s_{n+p}=\sum_{j=1}^{n+p}(-1)^{j}a_{j-1}=\sum_{j=1}^{p}(-1)^{j}a_{j-1}%
+\sum_{j=p+1}^{p+n}(-1)^{j}a_{j-1}=\sum_{j=p+1}^{p+n}(-1)^{j}a_{j-1}%
\]

If $p$ is even, then%
\[
\sum_{j=p+1}^{p+n}(-1)^{j}a_{j-1}=-a_{p}+a_{p+1}+\cdots+(-1)^{n+p}%
a_{n+p-1}=-a_{0}+a_{1}+\cdots+(-1)^{n}a_{n-1}=s_{n}%
\]

so
\[
t_{n+p}=e^{(-1)^{n+p}s_{n+p}}=e^{(-1)^{n}s_{n}}=t_{n}.
\]

If $p$ is odd, then%
\[
\sum_{j=p+1}^{p+n}(-1)^{j}a_{j-1}=a_{p}-a_{p+1}+\cdots+(-1)^{n+p}%
a_{n+p-1}=a_{0}-a_{1}+\cdots-(-1)^{n}a_{n-1}=-s_{n}%
\]

so%
\[
t_{n+p}=e^{(-1)^{n+p}s_{n+p}}=e^{-(-1)^{n}(-s_{n})}=e^{(-1)^{n}s_{n}}=t_{n}.
\]

and the proof is complete.
\end{proof}

\medskip

Note that the solution $\{t_{n}\}$ of (\ref{star1}) in Lemma \ref{bsc} need
not be bounded even if $\{a_{n}\}$ is a bounded sequence. The next result
expresses a useful fact for this case.

\begin{lemma}
\label{z}Assume that $\{a_{n}\}$ is bounded from above and $x_{0}, x_{-1}>0$.
If the sequence $\{t_{n}\}$ from $t_{0}$ given in (\ref{star1}) is unbounded
then some subsequence of the corresponding solution $\{x_{n}\}$ of (\ref{pce})
converges to 0.
\end{lemma}

\begin{proof}
By the hypotheses, $\sup_{n}a_{n}=a<\infty$ and there is a subsequence
$\{t_{n_{k}}\}$ such that $\lim_{k\rightarrow\infty}t_{n_{k}}=\infty.$ By
(\ref{star1}) and Lemma \ref{bsc}(b)%
\[
x_{n_{k}+1}\leq\frac{1}{e}t_{n_{k}+1}\leq\frac{e^{a-1}}{t_{n_{k}}}%
\]

Therefore,%
\[
\lim_{k\rightarrow\infty}x_{n_{k}+1}=\lim_{k\rightarrow\infty}\frac{e^{a-1}%
}{t_{n_{k}}}=0.
\]

\end{proof}

\section{The odd period case}

The dynamics of (\ref{pce}) depends critically on whether the period of the
parameter sequence $\{a_{n}\}$ is odd or even. In this section we consider the
odd case and the nature of solutions of (\ref{sc1}) in this case.

\begin{lemma}
\label{ptk}Suppose that $\{a_{n}\}$ is sequence of real numbers with minimal
odd period $p\geq1$ and let $\{t_{n}\}$ be a solution of (\ref{sc1}).

(a) $\{t_{n}\}$ has period $2p$ with a complete cycle $\{t_{0},t_{1}%
,\ldots,t_{2p-1}\}$ where $t_{k}$ is given by (\ref{gen1}) with
\begin{equation}
s_{k}=\left\{
\begin{array}
[c]{l}%
\sum_{j=1}^{k}(-1)^{j}a_{j-1},\quad\text{if }1\leq k\leq p\\
\sum_{j=k}^{2p-1}(-1)^{j}a_{j-p},\quad\text{if }p+1\leq k\leq2p-1
\end{array}
\right.  \label{sk}%
\end{equation}

(b) If $t_{0}=e^{-\sigma/2}$ then $\{t_{n}\}$ is periodic with period $p$.
\end{lemma}

\begin{proof}
(a) Let $\{a_{0},a_{1},\ldots,a_{p-1}\}$ be a full cycle of $a_{n}$ and define
$\sigma$ as in (\ref{sig}), i.e.
\[
\sigma=-a_{0}+a_{1}-a_{2}+\ldots-a_{p-1}.
\]

Since a full cycle of $a_{n}$ has an odd number of terms, expanding $s_{n}$ in
(\ref{gen2}) yields a sequence with alternating signs in terms of $\sigma$%
\[
s_{n}=\sigma-\sigma+\cdots+(-1)^{m-1}\sigma+(-1)^{m}\sum_{j=1}^{i}%
(-1)^{j}a_{j-1}%
\]
for integers $i,m$ such that $n=pm+i$, $m\geq0$ and $1\leq i\leq p$. If $m$ is
even then for $i=1,2,\ldots,p$%
\[
s_{n}=\sum_{j=1}^{i}(-1)^{j}a_{j-1}=\left\{
\begin{array}
[c]{ll}%
-a_{0} & n=pm+1\quad\text{(odd)}\\
-a_{0}+a_{1} & n=pm+2\quad\text{(even)}\\
\vdots & \vdots\\
-a_{0}+a_{1}\ldots-a_{p-1} & n=pm+p\quad\text{(odd)}%
\end{array}
\right.
\]

Similarly, if $m$ is odd then for $i=1,2,\ldots,p$%
\[
s_{n}=\sigma-\sum_{j=0}^{i}(-1)^{j}a_{j}=\left\{
\begin{array}
[c]{ll}%
\sigma+a_{0} & n=pm+1\quad\text{(even)}\\
\sigma+a_{0}-a_{1} & n=pm+2\quad\text{(odd)}\\
\vdots & \vdots\\
\sigma+a_{0}-a_{1}+\ldots-a_{p-1} & n=pm+p\quad\text{(even)}%
\end{array}
\right.
\]

The above list repeats for every consecutive pair of values of $m$ and yields
a complete cycle for $\{s_{n}\}.$ In particular, for $m=0$ we obtain for
$i=1,2,\ldots,p$%
\[
s_{n}=\sum_{j=1}^{i}(-1)^{j}a_{j-1}=\left\{
\begin{array}
[c]{ll}%
-a_{0} & n=1\\
-a_{0}+a_{1} & n=2\\
\vdots & \vdots\\
-a_{0}+a_{1}\ldots-a_{p-1} & n=p
\end{array}
\right.
\]
and for $m=1$ we obtain for $i=1,2,\ldots,p-1$%
\begin{align*}
s_{n}  &  =\sigma-\sum_{j=0}^{i}(-1)^{j}a_{j}=\left\{
\begin{array}
[c]{ll}%
a_{1}-a_{2}+\ldots-a_{p-1} & n=p+1\\
-a_{2}+\ldots-a_{p-1} & n=p+2\\
\vdots & \vdots\\
-a_{p-1} & n=2p-1
\end{array}
\right. \\
&  =\sum_{j=p+1}^{2p-1}(-1)^{j}a_{j-p}%
\end{align*}

This proves the validity of (\ref{sk}) and shows that the sequence $\{s_{n}\}$
has period $2p$. Now (\ref{gen1}) implies that $\{t_{n}\}$ also has period
$2p$ as claimed.

(b) If $\sigma=0$ then the statement follows immediately from Lemma
\ref{genper}. If $\sigma\not =0$ and $p$ is odd then%
\[
t_{p}=t_{0}^{(-1)^{p}}e^{(-1)^{p}s_{p}}=e^{\sigma/2-\sigma}=e^{-\sigma
/2}=t_{0}%
\]

In the proof of Lemma \ref{genper} it was shown that $s_{n+p}=\sigma-s_{n}$.
Thus%
\begin{align}
t_{n+p}  &  =t_{0}^{(-1)^{n+p}}e^{(-1)^{n+p}s_{n+p}}\nonumber\\
&  =e^{(-1)^{n}\sigma/2}e^{-(-1)^{n}(\sigma-s_{n}})\nonumber\\
&  =e^{-(-1)^{n}\sigma/2+(-1)^{n}s_{n}}=t_{0}^{(-1)^{n}}e^{(-1)^{n}s_{n}%
}=t_{n}\nonumber
\end{align}
and the proof is complete.
\end{proof}

\medskip

For $p=1$, Lemma \ref{ptk} implies that$\ \{t_{n}\}$ is the two-cycle%
\[
\left\{  t_{0},\frac{e^{a}}{t_{0}}\right\}
\]
where $a$ is the constant value of the sequence $\{a_{n}\}$. For $p=3$,
$\{t_{n}\}$ is the six-cycle%
\[
\left\{  t_{0},\frac{e^{a_{0}}}{t_{0}},t_{0}e^{a_{1}-a_{0}},\frac
{e^{a_{2}-a_{1}+a_{0}}}{t_{0}},t_{0}e^{a_{1}-a_{2}},\frac{e^{a_{2}}}{t_{0}%
}\right\}  .
\]

From the cofactor equation (\ref{star2}) we obtain%
\begin{align*}
x_{2n+2}  &  =t_{2n+2}x_{2n+1}e^{-x_{2n+1}}=t_{2n+2}t_{2n+1}x_{2n}\exp
(-x_{2n}-t_{2n+1}x_{2n}e^{-x_{2n}})\\
x_{2n+1}  &  =t_{2n+1}x_{2n}e^{-x_{2n}}=t_{2n+1}t_{2n}x_{2n-1}\exp\left(
-x_{2n-1}-t_{2n}x_{2n-1}e^{-x_{2n-1}}\right)
\end{align*}

For every solution $\{t_{n}\}$ of (\ref{star1}), $t_{n+1}t_{n}=e^{a_{n}}$ for
all $n$, so the even terms of the sequence $\{x_{n}\}$ satisfy%
\begin{equation}
x_{2n+2}=x_{2n}\exp\left(  a_{2n+1}-x_{2n}-t_{2n+1}x_{2n}e^{-x_{2n}}\right)
\label{reven}%
\end{equation}
and the odd terms satisfy%
\begin{equation}
x_{2n+1}=x_{2n-1}\exp\left(  a_{2n}-x_{2n-1}-t_{2n}x_{2n-1}e^{-x_{2n-1}%
}\right)  \label{rodd}%
\end{equation}

To reduce the notational clutter, let%
\begin{equation}
y_{n}=x_{2n}\;\;\;\rho_{n}=a_{2n+1}\;\;\;\mu_{n}=t_{2n+1} \label{yyn}%
\end{equation}
for $n\geq0$ and also
\begin{equation}
z_{n}=x_{2n-1}\;\;\;\zeta_{n}=a_{2n}\;\;\;\eta_{n}=t_{2n}. \label{zzn}%
\end{equation}

Then we can write (\ref{reven}) and (\ref{rodd}) as%

\begin{equation}
y_{n+1}=y_{n}e^{\rho_{n}-y_{n}-\mu_{n}y_{n}e^{-y_{n}}} \label{yn}%
\end{equation}

\begin{equation}
z_{n+1}=z_{n}e^{\zeta_{n}-z_{n}-\eta_{n}z_{n}e^{-z_{n}}} \label{zn}%
\end{equation}

The next result establishes the existence of an attracting, invariant interval
for (\ref{reven}) and (\ref{rodd}), or equivalently, (\ref{yn}) and (\ref{zn}).

\begin{lemma}
\label{invar} Let $\{a_{n}\}$ be a bounded sequence where $\inf_{n\geq0}%
a_{n}\in(0,2)$. Let $x_{0}, x_{-1}>0$ and $t_{0}$ be given as in (\ref{star1}.
Assume that the sequence $\{t_{2n+1}\}$ (respectively, $\{t_{2n}\}$) is
bounded and let $\{x_{n}\}$ be the corresponding solution of (\ref{pce}).

(a) There exists an interval $[\alpha,\beta]$ with $\alpha>0$ such that if
$x_{-1},x_{0}\in\lbrack\alpha,\beta]$ then $x_{2n}\in\lbrack\alpha,\beta]$
(respectively, $x_{2n+1}\in\lbrack\alpha,\beta]$) for $n\geq1$.

(b) For all $x_{0},x_{-1}>0$ there exists an integer $N>0$ such that
$x_{2n}\in\lbrack\alpha,\beta]$ (respectively, $x_{2n+1}\in\lbrack\alpha
,\beta]$) for all $n\geq N$.
\end{lemma}

\begin{proof}
(a) First, note that if $x_{0},x_{-1}>0$ then $x_{n}>0$ for all $n$ and by
Lemma \ref{bd} $x_{n}\leq e^{a-1}$ for $n\geq1$ where
\[
a=\sup_{n\geq0}a_{n}.
\]

Thus if%
\[
\beta=e^{a-1}%
\]
then $x_{n}\leq\beta$ for all $n$. Next, let%
\[
\rho=\inf_{n\geq0}a_{n}\in(0,2)
\]
and consider the map
\[
f(x)=xe^{\rho-x-\gamma xe^{-x}}%
\]
where $\gamma>0$ is fixed. Now $x^{\ast}$ is a fixed point of $f$ if and only
if
\[
x^{\ast}=f(x^{\ast})=x^{\ast}e^{\rho-x^{\ast}-\gamma x^{\ast}e^{-x^{\ast}}}%
\]
which is true if and only if $\rho-x^{\ast}-\gamma x^{\ast}e^{-x^{\ast}%
}=h(x^{\ast})=0$. Since $h(0)=\rho>0$ and $h(\rho)=-\gamma\rho e^{-\rho}<0$,
there is $x^{\ast}\in(0,\rho)$ such that $h(x^{\ast})=0$. Thus $f$ has a fixed
point $x^{\ast}\in(0,\rho)$. Further $f(x)>x$ for $x\in(0,x^{\ast})$ and
$f(x)<x$ for $x\in(x^{\ast},\beta)$. If
\[
\alpha=\min\{x^{\ast},f(\beta),f(1)\}.
\]
then we now show that $[\alpha,\beta]$ is invariant under $f$, i.e.
$f(x)\in\lbrack\alpha,\beta]$ for all $x\in\lbrack\alpha,\beta].$ There are
two possible cases:

\textit{Case 1}: $\gamma\leq e$. In this case, $f(x)$ has one critical point
at $x=1$ and it is increasing in $(0,1)$ and decreasing on $(1,\infty)$.\ Thus
$f(1)$ is a global maximum and thus $\alpha\neq f(1)$. First, consider the
case where $x^{\ast}<f(\beta)<\beta$ and let $x\in\lbrack x^{\ast},\beta]$. If
$x<1$, then $f(x)>f(x^{\ast})=x^{\ast}\geq\alpha$ because $f$ is increasing on
$(0,1)$. If $x>1$, then $f(x)>f(\beta)>x^{\ast}\geq\alpha$, because $f$ is
decreasing on $(1,\beta)$. In either case $f(x)\in\lbrack\alpha,\beta]$.

Next, consider the case where $f(\beta)<x^{\ast}<\beta$ and let $x\in\lbrack
f(\beta),\beta]$. Then $f(x)>x>f(\beta)=\alpha$ for $f(\beta)<x<x^{\ast}$. On
the other hand, if $x^{\ast}<x<1$, then $f(x)>f(x^{\ast})>f(\beta)=\alpha$ and
if $x^{\ast}<1<x<\beta$, then $f(x)>f(\beta)=\alpha$. It follows that
$f(x)\in\lbrack\alpha,\beta]$ if $\gamma\leq e.$

\textit{Case 2}: $\gamma>e$. In this case, $f(x)$ has three critical points
$x^{\prime},1$ and $x^{\prime\prime}$ with $x^{\prime}<1<x^{\prime\prime}$,
where local maxima occur at $x^{\prime}$ and $x^{\prime\prime}$and a local
minimum at 1. There are three possibilities:

(i) $\alpha=x^{\ast}$. In this case, for $x^{\ast}\leq x\leq x^{\prime}$,
$f(x)\geq f(x^{\ast})=x^{\ast}=\alpha$, since $f$ is increasing on
$(0,x^{\prime})$. If $x\in(x^{\prime},x^{\prime\prime})$, then $f(x)\geq
f(1)\geq\alpha$. If $\beta\geq x^{\prime\prime}$ and $x^{\prime\prime}\leq
x\leq\beta$ then $f(x)\geq f(\beta)\geq\alpha$, since $f$ is decreasing on
$(x^{\prime\prime},\infty)$.

(ii) $\alpha=f(\beta)$. In this case, for $x\in\lbrack f(\beta),x^{\ast})$,
$f(x)>x\geq f(\beta)=\alpha$. If $x^{\ast}\leq x^{\prime}$ and $x\in\lbrack
x^{\ast},x^{\prime}]$ then $f(x)\geq f(x^{\ast})=x^{\ast}\geq\alpha$ since $f$
is increasing. If $x\in(x^{\prime},x^{\prime\prime})$ then $f(x)\geq
f(1)\geq\alpha$. If $\beta\geq x^{\prime\prime}$ and $x^{\prime\prime}\leq
x\leq\beta$ then $f(x)\geq f(\beta)=\alpha$ since $f$ is decreasing.

(iii) $\alpha=f(1)$. In this case, if $x\in\lbrack f(1),x^{\ast})$ then
$f(x)>x>f(1)=\alpha$. If $x^{\ast}<x^{\prime}$ and $x\in\lbrack x^{\ast},1]$
then $f(x)\geq f(x^{\ast})=x^{\ast}\geq\alpha$ for $x\in\lbrack x^{\ast
},x^{\prime})$ and $f(x)\geq f(1)=\alpha$ for $x\in\lbrack x^{\prime},1]$. On
the other hand, if $x^{\ast}\geq x^{\prime}$ then $f(x)\geq f(1)=\alpha$ for
$x\in\lbrack x^{\ast},1]$. Finally, if $\beta>1$ and $x\in(1,\beta]$ then
$f(x)>f(1)=\alpha$ for $x\in(1,x^{\prime\prime})$ since $f$ is increasing on
$(1,x^{\prime\prime})$, and $f(x)\geq f(\beta)\geq\alpha$ for $x\in
(x^{\prime\prime},\beta]$ if $\beta>x^{\prime\prime}$.

The above three cases exhaust all possibilities so $f(x)\in\lbrack\alpha
,\beta]$ if $\gamma>e.$

Next, assume that $\{t_{2n+1}\}$ is bounded and let $\{y_{n}\}$ be as defined
by (\ref{yn}). If
\[
\gamma=\sup\{t_{2n+1}\}+1<\infty
\]
and $y_{n}\in\lbrack\alpha,\beta]$ then%
\[
y_{n+1}=y_{n}e^{a_{2n+1}-y_{n}-t_{2n+1}y_{n}e^{-y_{n}}}>y_{n}e^{\rho
-y_{n}-\gamma y_{n}e^{-y_{n}}}=f(y_{n})\geq\alpha
\]

Similarly, if $\{t_{2n}\}$ is bounded and $\gamma=\sup\{t_{2n}\}+1$ then
\[
z_{n+1}=z_{n}e^{a_{2n+2}-z_{n}-t_{2n+2}z_{n}e^{-z_{n}}}>z_{n}e^{\rho
-z_{n}-\gamma z_{n}e^{-z_{n}}}=f(z_{n})\geq\alpha
\]
which proves (a).

(b) It suffices to consider the case where $z_{n},y_{n}<\alpha$. We will do
this for $z_{n}$, since the case for $y_{n}$ can be done similarly. Let
\[
\tau=\sup\{t_{2n}\}+\frac{1}{2}=\gamma-\frac{1}{2}>0
\]
so for $x<x^{\ast}$ and $n\geq0$
\[
e^{a_{2n}-x-t_{2n}xe^{-x}}>e^{\rho-x-\tau xe^{-x}}>e^{\rho-x^{\ast}-\tau
x^{\ast}e^{-x^{\ast}}}>e^{\rho-x^{\ast}-\gamma x^{\ast}e^{-x^{\ast}}}=1
\]

Define
\[
k=e^{\rho-x^{\ast}-\tau x^{\ast}e^{-x^{\ast}}}>1.
\]

If $z_{n}<\alpha\leq x^{\ast}$, then
\[
z_{n+1}=z_{n}e^{a_{2n+2}-z_{n}-t_{2n+2}z_{n}e^{-z_{n}}}>z_{n}e^{\rho-x^{\ast
}-\tau x^{\ast}e^{-x^{\ast}}}=kz_{n}%
\]

If $z_{n+1}>\alpha$ then we're done; otherwise,
\[
z_{n+2}=z_{n+1}e^{a_{2n+4}-z_{n+1}-t_{2n+4}z_{n+1}e^{-z_{n+1}}}>z_{n+1}%
e^{\rho-x^{\ast}-tx^{\ast}e^{-x^{\ast}}}=kz_{n+1}=k^{2}z_{n}%
\]
and we continue in this way inductively. Since $k>1$ it follows that
$z_{n+N}>z_{n}k^{N}>\alpha$ for sufficiently large $N$.
\end{proof}

\medskip

\begin{lemma}
\label{deriv} Let $\{a_{n}\}$ be periodic of period $p$ and $0<a_{n}<2.$ If as
noted above, $\{y_{n}\}$ and $\{z_{n}\}$ are the even and odd indexed terms of
the solution $\{x_{n}\}$ of (\ref{pce}) with initial values $x_{0},x_{-1}%
\in\lbrack\alpha,\beta]$ then there are constants $K>0$ and $\delta\in(0,1)$
such that
\begin{equation}
\left\vert \prod_{i=0}^{n-1}(1-y_{i})(1-z_{i})\right\vert \leq K\delta^{n}
\label{prod}%
\end{equation}

\end{lemma}

\begin{proof}
Recall that if $g$ is a continuous function on the compact interval
$[\alpha,\beta]$ with $|g(x)|<1$ for all $x\in\lbrack\alpha,\beta]$ then by
the extreme value theorem there is a point $\tilde{x}\in\lbrack\alpha,\beta]$
such that $|g(x)|\leq|g(\tilde{x})|<1$ for $x\in\lbrack\alpha,\beta].$ Thus if
$\delta=|g(\tilde{x})|\in(0,1)$ then $|g(x)|\leq\delta$ for all $x\in
\lbrack\alpha,\beta]$.

Now we establish the inequality in (\ref{prod}). First, if $a=\max_{0\leq
i\leq p-1}\{a_{i}\}<1+\ln2$ then $\beta=e^{a-1}<2.$ Thus if $u_{i}$ denotes
either $y_{i}$ or $z_{i}$ then $u_{i}\in(0,2)\supset\lbrack\alpha,\beta]$,
i.e. $|1-u_{i}|<1$ and there exists $\delta_{1}\in(0,1)$ so that
$|1-u_{i}|<\delta_{1}$ for $u_{i}\in\lbrack\alpha,\beta]$.

Next, suppose that $a\geq1+\ln2$ and let
\[
2\leq u_{i}\leq e.
\]

Consider the preimage $u_{i-1}$ of $u_{i}$. There are two possible cases:
Either $u_{i-1}\leq1$ or $u_{i-1}\geq1$. \newline

\textit{Case 1}: If $u_{i-1}\leq1$ then%
\begin{align}
|1-u_{i-1}||1-u_{i}|  &  =(1-u_{i-1})(u_{i}-1)\nonumber\\
&  \leq(1-u_{i-1})(u_{i-1}e^{a-u_{i-1}-\tau_{i-1}u_{i-1}e^{-u_{i-1}}%
}-1)\nonumber\\
&  <(1-u_{i-1})(u_{i-1}e^{2-u_{i-1}}-1)\nonumber
\end{align}
where $\tau_{i}=\mu_{i}$ or $\eta_{i}$ depending on the case ($y_{n}$ or
$z_{n}$ respectively). Note that
\begin{equation}
(1-x)(xe^{2-x}-1)<1 \label{inq1}%
\end{equation}
for $x\in(0,1]$ because (\ref{inq1}) can be written as $x(1-x)<(2-x)e^{x-2}$
and this inequality is true since its left hand side has a maximum of 1/4 on
(0,1] whereas its right hand side has a minimum of $2e^{-2}>1/4$ on (0,1]. In
particular, (\ref{inq1}) holds for $x\in\lbrack\alpha,1]$ so there exists
$\delta_{2}\in(0,1)$ such that%
\[
|1-u_{i-1}||1-u_{i}|<\delta_{2}.
\]

\textit{Case 2}: If $u_{i-1}\geq1$ then%
\begin{align}
|1-u_{i-1}||1-u_{i}|  &  =(u_{i-1}-1)(u_{i}-1)\nonumber\\
&  \leq(u_{i-1}-1)(u_{i-1}e^{a-u_{i-1}-\tau_{i-1}u_{i-1}e^{-u_{i-1}}%
}-1)\nonumber\\
&  <(u_{i-1}-1)(u_{i-1}e^{a-u_{i-1}}-1).\nonumber
\end{align}

If $\phi(x)=(x-1)(xe^{2-x}-1)$ then%
\[
\phi^{\prime}(x)=[x-(x-1)^{2}]e^{2-x}-1,\quad\phi^{\prime\prime}%
(x)=(x-1)(x-4)e^{2-x}.
\]

Since $\phi$ is smooth with $\phi^{\prime}(2)=0$ and $\phi^{\prime\prime
}(x)<0$ for $x\in(1,4)$ it follows that $\phi$ is maximized on [1,4] at 2 and
$\phi(2)=1.$ In particular, for $u_{i-1}\in\lbrack1,\beta]\subset\lbrack
1,4]$,
\[
(u_{i-1}-1)(u_{i-1}e^{a-u_{i-1}}-1)<\phi(u_{i-1})\leq1
\]
and it follows that there is $\delta_{3}\in(0,1)$ such that
\[
|1-u_{i-1}||1-u_{i}|<\delta_{3}.
\]

Finally, there are at most $m$ pairings $|1-u_{i-1}||1-u_{i}|$ where $m=[n/2]$
(i.e. $m$ is $n/2$ rounded down to the nearest integer). If $n$ is even, then
$m=n/2$, if $n$ is odd, $m=(n-1)/2$ and we have one last unpaired term left,
namely, $|1-u_{0}|<(e-1)$. Choosing $\delta=\max\{\delta_{1},\delta_{2}%
,\delta_{3}\}$, we get%
\[
\left\vert \prod_{i=0}^{n-1}(1-u_{i})\right\vert <(e-1)\delta^{m}.
\]

Therefore,
\[
\left\vert \prod_{i=0}^{n-1}(1-y_{i})(1-z_{i})\right\vert <(e-1)^{2}%
\delta^{2m}\leq\frac{\beta}{\alpha}K\delta^{nm}%
\]
where $K=(e-1)^{2}/\delta$ and the proof is complete.
\end{proof}

\medskip

The next result generalizes similar results in \cite{FHL} and \cite{LS}.

\begin{theorem}
\label{per} Let $\{a_{n}\}$ be a periodic sequence with $0<a_{n}<2$,
$x_{0},x_{-1}>0$ and the sequence $\{t_{n}\}$ with $t_{0}=x_{0}/x_{-1}%
e^{{x_{-1}}}$ be periodic with period $q$. Then each solution of (\ref{pce})
from the initial values $x_{0},x_{-1}$ converges to a periodic solution
(dependent on the choice of initial values) with period $q$.
\end{theorem}

\begin{proof}
Let $q$ be the period of the sequence $\{t_{n}\}$ from initial value
$t_{0}=\frac{x_{0}}{x_{-1}e^{{x_{-1}}}}$. For each $i=1,2,\cdots,q$, define
the map
\[
g_{i}(x)=t_{i}xe^{-x}%
\]
and let
\[
\phi=g_{q}\circ g_{q-1}\circ\cdots\circ g_{1}%
\]
Then by the cofactor equation (\ref{star1}),
$\phi$ generates the orbit of $(1)$ from initial values $x_{0},x_{-1}$. Also
note that $\phi$ is an autonomous interval map, and by Lemma \ref{invar},
there exist real numbers $\alpha,\beta>0$ and a positive integer $N$ so that
$\phi:[\alpha,\beta]\rightarrow\lbrack\alpha,\beta]$ and $\phi^{n}%
(x)\in\lbrack\alpha,\beta]$ for all $n\geq N$. Hence, by Brouwer's fixed point
theorem, there exists a $x^{\ast}\in\lbrack\alpha,\beta]$ so that
$\phi(x^{\ast})=x^{\ast}$. Now, let $x_{0}\in\lbrack\alpha,\beta]$ be given.

\medskip Since
\[
g_{i}^{\prime}(x)=t_{i}e^{-x}(1-x)=\frac{g_{i}(x)}{x}(1-x)
\]
then
\begin{align}
\phi^{\prime}(x_{0}) & =\prod_{i=1}^{q}g_{i}^{\prime}(x_{i-1}) =\prod
_{i=1}^{q}\frac{g_{i}(x_{i-1})}{x_{i-1}}(1-x_{i-1})\nonumber\\
& =\frac{g_{1}(x_{0})}{x_{0}}\frac{g_{2}(x_{1})}{x_{1}}\cdots\frac
{g_{q}(x_{q-1})}{x_{q-1}}\prod_{i=1}^{q}(1-x_{i})\nonumber
\end{align}
Noting that $g_{i}(x_{i-1})=x_{i}$, we get
\[
\phi^{\prime}(x_{0})=\frac{x_{q}}{x_{0}}\prod_{i=1}^{q}(1-x_{i-1})
\]
Similarly,
\[
(\phi^{2})^{\prime}(x_{0})=(\phi\circ\phi)^{\prime}(x_{0})=\frac{x_{2q}}%
{x_{0}}\prod_{i=1}^{2q}(1-x_{i-1})
\]
and in general,
\[
(\phi^{n})^{\prime}(x_{0})=\frac{x_{nq}}{x_{0}}\prod_{i=1}^{nq}(1-x_{i-1})
\]

Now, let $m=[nq/2]$. If $nq$ is even, then $m=nq/2$ and by Lemma (\ref{deriv})

\[
\left\vert \prod_{i=1}^{nq}(1-x_{i})\right\vert =\left\vert \prod_{i=1}%
^{m}(1-y_{i})(1-z_{i})\right\vert \leq K\delta^{m}
\]
for some $K>0$, $\delta\in(0,1)$, where $y_{i}$ and $z_{i}$ are the even and
odd indexed terms of the $\{x_{n}\}$ as noted above. If $nq$ is odd, then
$m=(nq-1)/2$, so%

\[
\left\vert \prod_{i=1}^{nq}(1-x_{i})\right\vert =\left\vert (1-x_{nq}%
)\prod_{i=1}^{m}(1-y_{i})(1-z_{i})\right\vert \leq(e-1)K\delta^{m}
\]
hence,%

\[
\left\vert (\phi^{n})^{\prime}(x_{0})\right\vert \leq\frac{\alpha}{\beta
}(e-1)K\delta^{(nq-1)/2}.
\]

Finally,%
\[
|\phi^{n}(x_{0})-x^{\ast}|=|\phi^{n}(x_{0})-\phi^{n}(x^{\ast})|=|\left(
\phi^{n}\right)  ^{\prime}(w)||x_{0}-x^{\ast}|\leq\frac{\beta}{\alpha
}K(e-1)\delta^{(nq-1)/2}|x_{0}-x^{\ast}|\rightarrow0
\]
as $n\rightarrow\infty$ and the proof is complete.
\end{proof}

\medskip

The following result on the existence of periodic solutions is a consequence
of Theorem \ref{per}.

\begin{corollary}
\label{op}Let $\{a_{n}\}$ be periodic with minimal odd period p and further
assume that $0<a_{i}<2$ for $i=0,\cdots,p-1$.

(a) Each solution of (\ref{pce}) converges to a cycle with length $2p$ that
depends on the initial values $x_{-1},x_{0}>0;$

(b) If $x_{0}=x_{-1}e^{-\sigma/2-x_{-1}}$ (i.e. $t_{0}=e^{-\sigma/2}$), then
the solutions of (\ref{pce}) converge to a cycle of length $p$.
\end{corollary}

\begin{proof}
(a) By Lemma \ref{ptk}, $\{t_{n}\}$ is periodic with period $2p$, so
$\{a_{n}\}$ and $\{t_{n}\}$ have a common period $2p$. The rest follows from
Theorem \ref{per}.\newline

(b) If $t_{0}=e^{-\sigma/2}$, then by Lemma \ref{ptk} $t_{n}$ is periodic with
period $p$. Therefore, $\{a_{n}\}$ and $\{t_{n}\}$ have a common period $p$
and the rest follows from Theorem \ref{per}.
\end{proof}

\medskip

\begin{remark}
\label{ms}(Multistability) The periodic solutions in Corollary \ref{op} may be
distinct if the initial values are distinct, since the solutions of equations
(\ref{yn}) and (\ref{zn}) depend on the sequences $\{t_{2n+1}\}$ and
$\{t_{2n}\}$ which in turn depend on $t_{0}=x_{0}/(x_{-1}e^{-x_{-1}}).$ Thus
the cycles in Corollary \ref{op} are not locally stable, hence not ordinary
limit cycles. Since cycles with different values of $t_{0}$ coexist, we see
that (\ref{pce}) exhibits multistability. Figures \ref{Fig1} and \ref{Fig2}
illustrate this situation for period $p=3$ with%
\[
a_{0}=1,\quad a_{1}=1.9,\quad a_{2}=0.8.
\]

\end{remark}

\begin{figure}[tbp] 
  \centering
  \includegraphics[width=5.02in,height=2.41in,keepaspectratio]{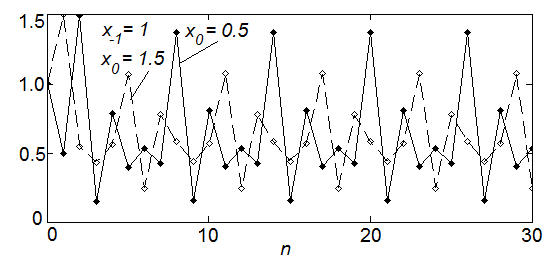}
  \caption{Coexisting period 6 solutions with parameter period p=3}
  \label{Fig1}
\end{figure}

\begin{figure}[tbp] 
  \centering
  \includegraphics[width=5.02in,height=2.49in,keepaspectratio]{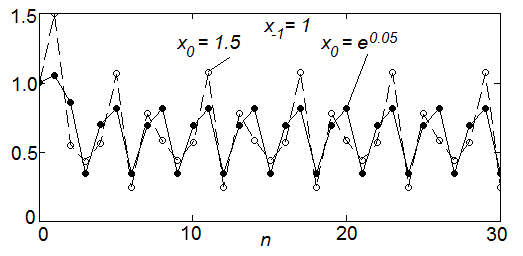}
  \caption{Coexisting period 6 and period 3 (exceptional) solutions, p=3}
  \label{Fig2}
\end{figure}

Next, we consider a wider range of values for $a_{n}$ and the existence of
non-periodic solutions for (\ref{pce}). The next two results are needed.

\begin{lemma}
\label{orb}Let $\{x_{n}\}$ be a solution of (\ref{pce}) with initial values
$x_{-1},x_{0}>0$ and assume that $\{a_{n}\}$ is periodic with minimal period
$p\geq1$ and $\{t_{n}\}$ is periodic with period $q\geq1$. Define%
\[
g_{k}(x)=t_{k}xe^{-x},\ k=0,1,\ldots,q-1
\]
where $t_{0}=x_{0}/(x_{-1}e^{-x_{-1}})$ and $t_{k}$ is given by (\ref{gen1})
and (\ref{gen2}). Also define%
\begin{align*}
h_{k}  &  =g_{k}\circ g_{k-1}\circ\cdots\circ g_{0},\quad k=0,1,\ldots,q-1\\
f  &  =h_{q-1}=g_{q-1}\circ g_{q-2}\circ\cdots\circ g_{1}\circ g_{0}%
\end{align*}

Then $\{x_{n}\}$ is determined by the $q$ sequences%
\begin{equation}
x_{qm+k}=h_{k}\circ f^{\,m}(x_{-1}),\quad k=0,1,\ldots,q-1 \label{oprk2}%
\end{equation}
that are obtained by iterations of one-dimensional maps of the interval
$(0,\infty)$, with $f^{\,0}$ being the identity map.
\end{lemma}

\begin{proof}
Given the initial values $x_{-1},x_{0}>0$ the definition of $t_{0}$ and
(\ref{star2}) imply that
\begin{align*}
x_{0}  &  =t_{0}x_{-1}e^{-x_{-1}}=g_{0}(x_{-1})=h_{0}(x_{-1})\\
x_{1}  &  =t_{1}x_{0}e^{-x_{0}}=g_{1}(x_{0})=g_{1}\circ g_{0}(x_{-1}%
)=h_{1}(x_{-1})
\end{align*}
and so on:%
\[
x_{k}=h_{k}(x_{-1}),\quad k=0,1,\ldots,q-2
\]

Thus (\ref{oprk2}) holds for $m=0.$ Further, $x_{q-1}=h_{q-1}(x_{-1}%
)=f(x_{-1})$.\ Inductively, we suppose that (\ref{oprk2}) holds for some
$m\geq0$ and note that for $k=0,1,\ldots,q-2$
\[
h_{k+1}=g_{k+1}\circ g_{k}\circ\cdots\circ g_{0}=g_{k+1}\circ h_{k}%
\]

Now by (\ref{star2})
\begin{align*}
x_{q(m+1)-1}  &  =t_{qm+q-1}x_{qm+q-2}e^{-x_{qm+q-2}}\\
&  =t_{q-1}h_{q-2}\circ f^{\,m}(x_{-1})e^{-h_{q-2}\circ f^{m}(x_{-1})}\\
&  =g_{q-1}\circ h_{q-2}\circ f^{\,m}(x_{-1})\\
&  =h_{q-1}\circ f^{\,m}(x_{-1})\\
&  =f^{\,m+1}(x_{-1})
\end{align*}

So (\ref{oprk2}) holds for $k=q-1$ by induction. Further, again by
(\ref{star2}) and the preceding equality
\begin{align*}
x_{q(m+1)}  &  =t_{qm+q}x_{qm+q-1}q^{-x_{qm+q-1}}\\
&  =t_{0}f^{m+1}(x_{-1})e^{-f^{m+1}(x_{-1})}\\
&  =g_{0}\circ f^{\,m+1}(x_{-1})\\
&  =h_{0}\circ f^{\,m+1}(x_{-1})
\end{align*}

Similarly,%
\begin{align*}
x_{q(m+1)+1}  &  =t_{q(m+1)+1}x_{q(m+1)}e^{-x_{q(m+1)}}\\
&  =t_{1}h_{0}\circ f^{m+1}(x_{-1})e^{-h_{0}\circ f^{m+1}(x_{-1})}\\
&  =g_{1}\circ h_{0}\circ f^{m+1}(x_{-1})\\
&  =h_{1}\circ f^{m+1}(x_{-1})
\end{align*}

Repeating this calculation $q-2$ times establishes (\ref{oprk2}) and completes
the induction step and the proof.

\medskip
\end{proof}

\begin{lemma}
\label{pq} Suppose that $\{a_{n}\}$ and $\{t_{n}\}$ are periodic and
$\{t_{n}\}$ has minimal period $q\geq1$.

(a) If the map $f$ in Lemma \ref{orb} has a (positive) periodic point of
minimal period $\omega$ then there is a solution of (\ref{pce}) with period
$\omega q.$

(b) If the map $f$ in Lemma \ref{orb} has a non-periodic point then
(\ref{pce}) has a non-periodic solution.
\end{lemma}

\begin{proof}
(a) By hypothesis, there is a number $s\in(0,\infty)$ such that $f^{n+\omega
}(s)=f^{n}(s)$ for all $n\geq0.$ Let $x_{-1}=s$ and define $x_{0}=h_{0}(s)$.
By Lemma \ref{orb} the solution $x_{n}$ corresponding to these initial values
follows the track shown below:%
\[%
\begin{array}
[c]{rrrr}%
x_{-1}=s\rightarrow & x_{0}=h_{0}(s)\rightarrow & \cdots\rightarrow &
x_{q-2}=h_{q-2}(s)\rightarrow\\
\rightarrow x_{q-1}=h_{q-1}(s)=f(s)\rightarrow & x_{q}=h_{0}(f(s))\rightarrow
& \cdots\rightarrow & x_{2q-2}=h_{q-2}(f(s))\rightarrow\\
\rightarrow x_{2q-1}=h_{q-1}(f(s))=f^{2}(s)\rightarrow & x_{2q}=h_{0}%
(f^{2}(s))\rightarrow & \cdots\rightarrow & x_{3q-2}=h_{3q-2}(f^{2}%
(s))\rightarrow\\
\multicolumn{1}{c}{\vdots} & \multicolumn{1}{c}{\vdots} &
\multicolumn{1}{c}{\vdots} & \multicolumn{1}{c}{\vdots}\\
x_{\omega q-1}=h_{q-1}(f^{\omega-1}(s))=f^{\omega}(s)=s\rightarrow &
x_{q\omega}=h_{0}(s)\rightarrow & \cdots\rightarrow & x_{(\omega
+1)q-2}=h_{q-2}(s)\rightarrow\cdots
\end{array}
\]

The pattern in this list evidently repeats after $\omega q$ entries. So
$x_{\omega q+n}=x_{n}$ for $n\geq0$ and it follows that the solution
$\{x_{n}\}$ of (\ref{pce}) has period $\omega q$.

(b) Suppose that $\left\{  f^{n}(x_{-1})\right\}  $ is a non-periodic sequence
for some $x_{-1}>0.$ Then by Lemma \ref{orb} the solution $\{x_{n}\}$ of
(\ref{pce}) with initial values $x_{-1}$ and $x_{0}=g_{0}(x_{-1})$ has the
non-periodic subsequence
\[
x_{qn-1}=f^{n}(x_{-1})
\]

It follows that $\{x_{n}\}$ is non-periodic.
\end{proof}

\medskip

The following involves a wider range of values for $a_{n}$ than Corollary
\ref{op} which in particular, allows for the existence of non-periodic solutions.

\begin{theorem}
\label{oddgen} Suppose that $\{a_{n}\}$ is periodic with minimal odd period
$p\geq1$ and let $f$ be the interval map in Lemma \ref{orb} where $t_{0}>0$ is
a fixed real number and $t_{k}$ is given by (\ref{gen1})-(\ref{gen2}) for
$k\geq1$.

(a) If $s$ is a periodic point of $f$ with period $\omega$ then all solutions
of (\ref{pce}) with initial values $x_{-1}=s$ and $x_{0}=t_{0}se^{-s}$ (i.e.
$(x_{-1},x_{0})$ is on the curve $g_{0}$) have period $2p\omega.$

(b) If $t_{0}=e^{-\sigma/2}$ and $s$ is a periodic point of $f$ with period
$\omega$, then all solutions of (\ref{pce}) with initial values $x_{-1}=s$ and
$x_{0}=se^{-\sigma/2-s}$ have period $p\omega$ (in this case the graph of
$g_{0}$ is an invariant set of (\ref{pce}) in the state-plane).

(c) If the map $f$ has a non-periodic point, then (\ref{pce}) has a
non-periodic solution.

(d) If $f$ has a period-three point then (\ref{pce}) has periodic solutions of
period $2pn$ for all positive integers $n$ as well as chaotic solutions in the
sense of Li-Yorke (\cite{El},\cite{LY}).
\end{theorem}

\begin{proof}
Parts (a) and (c) are immediate consequences of Lemma \ref{pq} with $q=2p$
because of Lemma \ref{ptk}(a). Part (b) is true by Lemma \ref{ptk}(b).

(d) As is well-known from \cite{LY}, if $f$ has a period three point then $f$
has periodic points of every period $n\geq1$, as well as aperiodic, chaotic
solutions in the sense of Li and Yorke. Therefore, by parts (a) and (b),
(\ref{pce}) also has periodic solutions of period $2pn$, as well as chaotic solutions.
\end{proof}

\medskip

In the case $p=1$, i.e. when (\ref{pce}) is autonomous with $a_{n}=a$ for all
$n,$ the conditions stated in Theorem \ref{oddgen} were examined in \cite{LS}.
In particular, it was verified that if $a\geq3.13$ then (\ref{pce}) has
chaotic solutions from certain initial conditions. The multistable nature of
solutions of (\ref{pce}) was also discussed in detail.

For odd $p\geq3$ the map $f$ is a composition of $2p$ functions and therefore,
analytically less tractable. We use numerical simulations to highlight the
rich variety of coexisting solutions that Theorem \ref{oddgen} allows. As
noted above, and explained in greater detail in \cite{LS}, these solutions are
attracting (though not locally stable) so they are observable and may be
recorded numerically.

In the next four figures, $p=3$ with%
\[
a_{0}=1,\quad a_{1}=2,\quad a_{2}=4.
\]

\begin{figure}[tbp] 
  \centering
  \includegraphics[width=5.67in,height=3.11in,keepaspectratio]{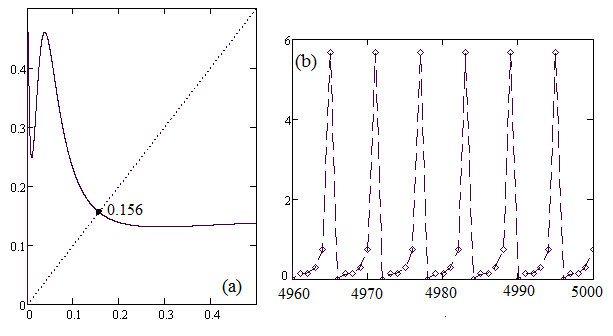}
  \caption{(a) Graph of $f$; (b) Corresponding 6-cycle}
  \label{Fig3}
\end{figure}

The initial values that generate the 6-cycle in Figure \ref{Fig3} are
$x_{-1}=1$ and $x_{0}=0.8.$ Panel (a) shows the map $f=g_{5}\circ\cdots\circ
g_{0}$ which is a composition of six exponential maps, together with a single
stable positive fixed point that corresponds to the solution of (\ref{pce})
shown in Panel (b).

\begin{figure}[tbp] 
  \centering
  \includegraphics[width=5.67in,height=2.97in,keepaspectratio]{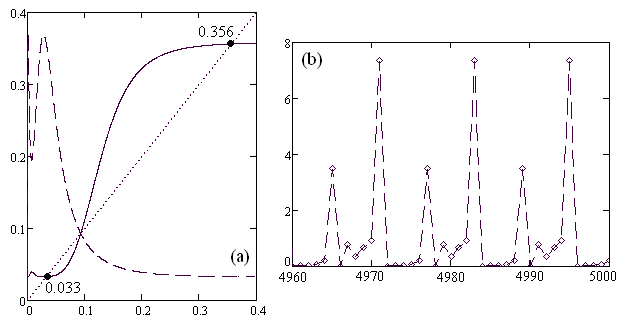}
  \caption{(a) Graphs of $f$ and $f^{2}$; (b) Corresponding 12-cycle}
  \label{Fig4}
\end{figure}

In Figure \ref{Fig4} the initial values are $x_{-1}=x_{0}=1.$ In Panel (a) the
graphs of $f$ and $f^{2}$ are shown that indicate the presence of a stable
2-cycle (the fixed point of $f$ is unstable in this case). Panel (b) shows the
corresponding 12-cycle for (\ref{pce}), as required by Theorem \ref{oddgen}
with $\omega=2$. We emphasize that this solution coexists stably with the
6-cycle in \ref{Fig3}. We further note 12-cycles do not exist under the
hypotheses of Corollary \ref{op} since those hypotheses imply the stability of
the fixed point of $f$.

\begin{figure}[tbp] 
  \centering
  \includegraphics[width=5.67in,height=3.82in,keepaspectratio]{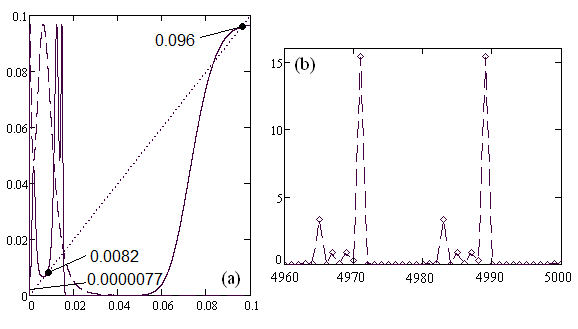}
  \caption{(a) Graphs of $f$ and $f^{3}$; (b) Corresponding 18-cycle}
  \label{Fig5}
\end{figure}

In Figure \ref{Fig5} the initial values are $x_{-1}=1$ and $x_{0}=3.8.$ In
Panel (a) the graphs of $f$ and $f^{3}$ are shown where a stable 3-cycle is
indicated that corresponds to the solution of (\ref{pce}) that is shown in
Panel (b). As stated in Theorem \ref{oddgen} this is an 18-cycle since now
$\omega=3$. This solution coexists stably with the 6-cycle and the 12-cycle
above. Also note that 18-cycles do not exist under the hypotheses of Corollary
\ref{op}.

\begin{figure}[tbp] 
  \centering
  \includegraphics[width=5.67in,height=2.73in,keepaspectratio]{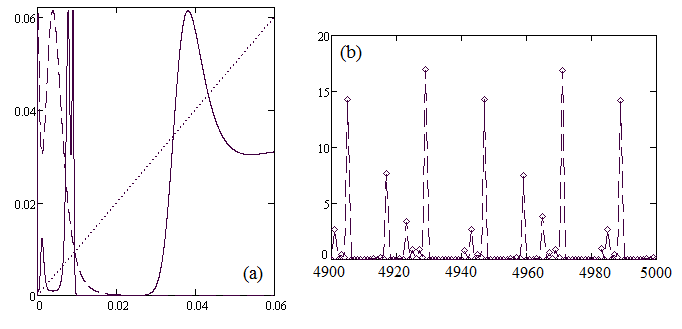}
  \caption{(a) Graphs of $f$ and $f^{3}$; (b) Corresponding nonperiodic solution}
  \label{Fig6}
\end{figure}

Finally, in Figure \ref{Fig6} the initial values are $x_{-1}=1$ and $x_{0}=6.$
Panel (a) shows the graphs of $f$ and $f^{3}$ where we can identify a pair of
unstable 3-cycles where the graph of $f^{3}$ crosses the identity line (in
addition to the unstable fixed point of $f$). The map $f$ then exhibits
Li-Yorke type chaos. A portion of the plot of the corresponding solution of
(\ref{pce}) is shown in Panel (b). This nonperiodic solution coexists stably
with the periodic solutions mentioned above. However, nonperiodic solutions do
not exist under the hypotheses of Corollary \ref{op}.

\section{The even period case}

When $\{a_{n}\}$ is periodic with minimal even period $p$ the next result
shows that the sequence $\{t_{n}\}$ is \textit{not} periodic with the
exception of a boundary case. This causes a fundamental change in the dynamics
of (\ref{pce}). Once again, the quantity $\sigma$ is defined by (\ref{sig}),
i.e.%
\[
\sigma=-a_{0}+a_{1}-a_{2}+\ldots+a_{p-1}.
\]

\begin{lemma}
\label{pet}Suppose that $\{a_{n}\}$ is a sequence of real numbers with minimal
even period $p\geq2$ and let $\{t_{n}\}$ be a solution of (\ref{sc1}). Then%
\begin{equation}
t_{n}=\left(  t_{0}e^{d_{n}\sigma+\gamma_{n}}\right)  ^{(-1)^{n}} \label{pe}%
\end{equation}
where the integer divisor $d_{n}=[n-n(\operatorname{mod}p)]/p$ is uniquely
defined by each $n$ and
\begin{equation}
\gamma_{n}=\left\{
\begin{array}
[c]{ll}%
\sum_{j=1}^{n(\operatorname{mod}p)}(-1)^{j}a_{j-1} & \text{if }%
n(\operatorname{mod}p)\not =0\\
0 & \text{if }n(\operatorname{mod}p)=0
\end{array}
\right.  \label{gamn}%
\end{equation}

The sequence $\{t_{n}\}$ is periodic with period $p$ iff $\sigma=0$, i.e.%
\begin{equation}
a_{0}+a_{2}+\cdots a_{p-2}=a_{1}+a_{3}+\cdots+a_{p-1}. \label{peq}%
\end{equation}

\end{lemma}

\begin{proof}
Let $\{a_{0},a_{1},\ldots,a_{p-1}\}$ be a full cycle of $a_{n}$ with an even
number of terms. Since $n=pd_{n}+n(\operatorname{mod}p)$ for $n\geq1$, expand
$s_{n}$ in (\ref{gen2}) to obtain%
\[
s_{n}=d_{n}\sigma+\sum_{j=1}^{n(\operatorname{mod}p)}(-1)^{j}a_{j-1}%
\]
if $n(\operatorname{mod}p)\not =0.$ If $p$ divides $n$ so that
$n(\operatorname{mod}p)=0$ then we assume that the sum is 0 and $s_{n}%
=d_{n}\sigma.$ Thus $s_{n}=d_{n}\sigma+\gamma_{n}$ where $\gamma_{n}$ is as
defined in (\ref{gamn}).

The $\sigma$ terms have uniform signs in this case since there are an even
number of terms in each full cycle of $a_{n}$. Now (\ref{gen1}) yields%
\[
t_{n}=t_{0}^{(-1)^{n}}e^{(-1)^{n}s_{n}}=t_{0}^{(-1)^{n}}e^{(-1)^{n}%
(d_{n}\sigma+\gamma_{n})}%
\]
which is the same as (\ref{pe}).

Next, if $\sigma\not =0$ then $d_{n}\sigma$ is unbounded as $n$ increases
without bound so $\{t_{n}\}$ is not periodic. But if $\sigma=0$ then
(\ref{pe}) reduces to%
\begin{equation}
t_{n}=\left(  t_{0}e^{\gamma_{n}}\right)  ^{(-1)^{n}} \label{pe0}%
\end{equation}

Since the sequence $\gamma_{n}$ has period $p,$ the expression on the right
hand side of (\ref{pe0}) has period $p$ with a full cycle%
\[
t_{1}=\frac{e^{a_{0}}}{t_{0}},t_{2}=t_{0}e^{-a_{0}+a_{1}},t_{3}=\frac
{e^{a_{0}-a_{1}+a_{2}}}{t_{0}},\ldots,t_{p}=t_{0}e^{-a_{0}+a_{1}%
+\cdots+(-1)^{p}a_{p-1}}=t_{0}.
\]

\end{proof}

\medskip

By the preceding result,
\begin{align*}
t_{2m}  &  =t_{0}e^{\gamma_{2m}}e^{d_{2m}\sigma}\qquad\quad\text{if
}n=2m\text{ is even}\\
t_{2m+1}  &  =\frac{1}{t_{0}}e^{-\gamma_{2m+1}}e^{-d_{2m+1}\sigma}%
\quad\text{if }n=2m+1\text{ is odd}%
\end{align*}

Suppose that $\sigma\not =0.$ If $\sigma>0$ then since $\lim_{n\rightarrow
\infty}d_{n}=\infty$ it follows that $t_{2m}$ is unbounded but $t_{2m+1}$
converges to 0, and the reverse is true if $\sigma<0.$ Therefore,
\begin{align}
\lim_{m\rightarrow\infty}t_{2m}  &  =\infty,\quad\lim_{m\rightarrow\infty
}t_{2m+1}=0,\quad\text{if }\sigma>0,\label{tn1}\\
\lim_{m\rightarrow\infty}t_{2m}  &  =0,\quad\lim_{m\rightarrow\infty}%
t_{2m+1}=\infty,\quad\text{if }\sigma<0. \label{tn2}%
\end{align}

\begin{lemma}
\label{1h}Suppose that $\{a_{n}\}$ is a sequence of real numbers with minimal
even period $p\geq2$ and let $\{x_{n}\}$ be a solution of (\ref{pce}) with
initial values $x_{-1},x_{0}>0.$ Then $\lim_{n\rightarrow\infty}x_{2n+1}=0$ if
$\sigma>0$ and $\lim_{n\rightarrow\infty}x_{2n}=0$ if $\sigma<0$.
\end{lemma}

\begin{proof}
Assume first that $\sigma>0$. Then by (\ref{tn1}) $\lim_{n\rightarrow\infty
}t_{2n}=\infty$ so as in the proof of Lemma \ref{z} $\lim_{n\rightarrow\infty
}x_{2n+1}=0.$ If $\sigma<0$ then a similar argument using (\ref{tn2}) yields
$\lim_{n\rightarrow\infty}x_{2n}=0$ to complete the proof.
\end{proof}

\medskip

Lemma \ref{1h} indicates that half of the terms of every solution $\{x_{n}\}$
of (\ref{pce}) converge to 0 in the even period case if $\sigma\not =0.$ We
now consider what happens to the \textit{other} half.

\begin{lemma}
\label{sol} Let $\{u_{n}\}$ be the solution of
\begin{equation}
u_{n+1}=u_{n}e^{a_{2n+1}-u_{n}} \label{fo}%
\end{equation}
and $\{w_{n}\}$ be the solution of
\begin{equation}
w_{n+1}=w_{n}e^{a_{2n+2}-w_{n}}. \label{fo1}%
\end{equation}

(a) The sequence $\{x_{n}\}$ with $x_{2n}=u_{n}$ and $x_{2n+1}=0$ is a
solution of (\ref{pce}).

(b) The sequence $\{x_{n}\}$ with $x_{2n}=0$ and $x_{2n+1}=w_{n}$ is a
solution of (\ref{pce}).
\end{lemma}

\begin{proof}
(a) Let $\{u_{n}\}$ be a solution to (\ref{fo}) from initial value $u_{0}>0$.
If $x_{0}=u_{0}$ and $x_{1}=0$, then
\[
x_{2}=x_{0}e^{a_{1}-x_{0}-x_{1}}=u_{0}e^{a_{1}-u_{0}}=u_{1}%
\]
and
\[
x_{3}=x_{1}e^{a_{2}-x_{2}-x_{1}}=0
\]
Inductively, if $x_{2k}=u_{k}$ and $x_{2k+1}=0$ for some $k\geq1$ then
\[
x_{2k+2}=x_{2k}e^{a_{2k+1}-x_{2k}-x_{2k+1}}=u_{k}e^{a_{2k+1}-u_{k}}=u_{k+1}%
\]
and
\[
x_{2k+3}=x_{2k+1}e^{a_{2k+2}-x_{2k+1}-a_{2k+2}}=0
\]
which proves (a).

\medskip(b) Let $\{w_{n}\}$ be a solution to (\ref{fo1}) from initial value
$w_{0}>0$. If $x_{0}=0$ and $x_{1}=w_{0}$, then
\[
x_{2}=x_{0}e^{a_{1}-x_{0}-x_{1}}=0
\]
and
\[
x_{3}=x_{1}e^{a_{2}-x_{2}-x_{1}}=w_{0}e^{a_{2}-w_{0}}=w_{1}
\]
Inductively, if $x_{2k}=0$ and $x_{2k+1}=w_{k}$ for some $k\geq1$ then
\[
x_{2k+2}=x_{2k}e^{a_{2k+1}-x_{2k}-x_{2k+1}}=0
\]
and
\[
x_{2k+3}=x_{2k+1}e^{a_{2k+2}-x_{2k+1}-a_{2k+2}}=w_{k}e^{a_{2k+2}-w_{k}%
}=w_{k+1}
\]
which proves (b).
\end{proof}

\medskip

The next result is proved in \cite{sacker}.\ 

\begin{lemma}
\label{skr}Consider the first-order difference equation
\begin{equation}
y_{n+1}=y_{n}e^{\alpha_{n}-y_{n}} \label{sack}%
\end{equation}
where $\alpha_{n}$ is a sequence of real numbers with period $q$. If
$0<\alpha_{n}<2$ then (\ref{sack}) has a globally asymptotically stable
solution $\{y_{n}^{\ast}\}$ with period $q$ such that%
\[
\sum_{i=1}^{q}y_{i}^{\ast}=\sum_{i=1}^{q}\alpha_{i}.
\]

\end{lemma}

\begin{theorem}
\label{peven}Let $\{a_{n}\}$ be periodic with minimal even period $p\geq2$ and
let $\sigma$ be as defined in (\ref{sig}).

(a) If $\sigma>0$ and $0<a_{2k-1}<2$ for $k=1,2,\ldots p/2$ then (\ref{pce})
has a globally attracting periodic solution $\{\bar{x}_{n}\}$ with period $p$
such that $\bar{x}_{2n-1}=0$ and $\bar{x}_{2n}$ is a sequence of period $p/2$
satisfying the equality%
\[
\sum_{i=1}^{p/2}\bar{x}_{2i-2}=\sum_{i=1}^{p/2}a_{2i-1}.
\]

(b) If $\sigma<0$ and $0<a_{2k-2}<2$ for $k=1,2,\ldots p/2$ then (\ref{pce})
has a globally attracting periodic solution $\{\bar{x}_{n}\}$ with period $p$
such that $\bar{x}_{2n}=0$ and $\bar{x}_{2n-1}$ is a sequence of period $p/2$
satisfying the equality%
\[
\sum_{i=1}^{p/2}\bar{x}_{2i-1}=\sum_{i=1}^{p/2}a_{2i-2}.
\]

\end{theorem}

\begin{proof}
We prove part (a) and part (b) is demonstrated similarly. By Lemma \ref{skr}
the equation in (\ref{fo}) has a periodic solution of period $p/2$ given by
$\{u_{i}^{\ast}\}$ with $0\leq i\leq p/2-1$. By Lemma \ref{sol}, the sequence
$\{u_{0}^{\ast},0,u_{1}^{\ast},0,\cdots,u_{p/2-1}^{\ast},0\}$ is a $p$
periodic solution of (\ref{pce}) This means that $\bar{x}_{2n-1}=0$ and
$\bar{x}_{2n}=u_{i}^{\ast}$ with%
\[
\sum_{i=1}^{p/2}\bar{x}_{2i-2}=\sum_{i=1}^{p/2}u_{i}^{\ast}=\sum_{i=1}%
^{p/2}a_{2i-1}.
\]

Let the even indexed terms of the solution $\{x_{n}\}$ be defined as in
(\ref{yn}) and for each $n\geq0$, define
\[
F_{n}(x)=xe^{\rho_{n}-x-\mu_{n}xe^{-x}}%
\]

Then $F_{n}(u_{n}^{\ast})=u_{n+1}^{\ast}$. Now observe that with $\xi
_{n}=F_{n}\circ F_{n-1}\circ\cdots\circ F_{0}$
\[
F_{n}(y_{n})=F_{n}(F_{n-1}(y_{n-1}))=F_{n}(F_{n-1}(\cdots F_{0}(y_{0}%
))\cdots)=\xi_{n}(y_{0})
\]

Also note that%
\[
|\xi_{n}^{\prime}(y_{0})|=\left\vert \prod_{i=0}^{n}e^{\rho_{i}-y_{i}-\mu
_{i}y_{i}e^{-y_{i}}}(1-\mu_{i}y_{i}e^{-y_{i}})(1-y_{i})\right\vert
\]

Since $\mu_{n}\rightarrow0$, for sufficiently large $N$, $0<(1-\mu_{n}%
y_{n}e^{-y_{n}})\leq1$ for $n\geq N$. Then there exists a constant $M>0$ so
that%
\[
\left\vert \prod_{i=0}^{n}(1-\mu_{i}y_{i}e^{-y_{i}})\right\vert \leq\left\vert
\prod_{i=0}^{N}(1-z_{i+1})\right\vert \leq M
\]

Proceeding now as in the proof of Lemma \ref{deriv}, if we let $m=[n/2]$, we
can find constants $K>0$ and $\delta\in(0,1)$ so that%
\[
\left\vert \prod_{i=0}^{n}(1-y_{i})\right\vert \leq K\delta^{m}%
\]

Therefore,%
\[
|\xi_{n}^{\prime}(y_{0})|=\frac{y_{n+1}}{y_{0}}\left\vert \prod_{i=0}%
^{n}(1-\mu_{i}y_{i}e^{-y_{i}})(1-y_{i})\right\vert \leq\frac{\beta}{\alpha
}KM\delta^{m}%
\]

Finally,
\[
|y_{n+1}-u_{n+1}^{\ast}|=|F_{n}(y_{n})-F_{n}(u_{n}^{\ast})|=|\xi_{n}%
(y_{0})-\xi_{n}(u_{0}^{\ast})|=|\xi^{\prime}(w)||y_{0}-u_{0}^{\ast}|\leq
\frac{b}{\alpha}KM\delta^{m}|y_{0}-u_{0}^{\ast}|\rightarrow0
\]
as $n\rightarrow\infty$ which completes the proof.
\end{proof}

\medskip

\begin{remark}
1. In Theorem \ref{peven}(a) the even-indexed terms $a_{2k}$ are not
restricted to (0,2) as long as $\sigma>0$, i.e.%
\[
a_{1}+a_{3}+\cdots+a_{p-1}>a_{0}+a_{2}+\cdots+a_{p-2}%
\]

This imposes an upper bound $a_{2k}<2(p/2)=p$ for each $k$ but clealy some
$a_{2k}$ may exceed 2. Similarly, in (b) the odd-indexed terms are not
restricted to (0,2) as long as $\sigma<0$.

2. Note that $2p$ is not a minimal period for $\{\bar{x}_{n}\}.$ For example,
if $p=4$ with $a_{1}=a_{3}$ and $2a_{1}>a_{0}+a_{2}$ (so that $\sigma>0$) then
$\bar{x}_{2n-1}$ satisfies (\ref{sack}) with constant $\rho_{n}.$ In this
case, Lemma \ref{skr} yields a globally asymptotically stable fixed point for
(\ref{sack}), and thus a globally attracting period two solution for
(\ref{pce}).
\end{remark}

Figure \ref{Fig7} illustrates Theorem \ref{peven} with $p=4$ and 
$$a_{0}=1.4,\quad a_{1}=1.8,\quad a_{2}=1.6,\quad a_{3}=0.3$$

\begin{figure}[tbp] 
  \centering
  \includegraphics[width=5.02in,height=2.64in,keepaspectratio]{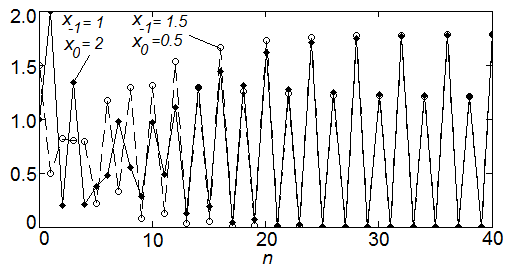}
  \caption{Solutions converging to a single 4-cycle with parameter period p=4}
  \label{Fig7}
\end{figure}

In the boundary special case $\sigma=0$, the solutions of (\ref{pce}) have entirely
different dynamics that resemble the odd period case. Indeed, the next result
is similar to Theorem \ref{oddgen}.

\begin{corollary}
Suppose that $\{a_{n}\}$ is periodic with minimal even period $p$ and
$\sigma=0.$ Let $f$ be the interval map in Lemma \ref{orb} where $t_{0}>0$ is
a fixed real number and $t_{k}$ is given by (\ref{gen1})-(\ref{gen2}) for
$k\geq1$.

(a) If $s$ is a periodic point of $f$ with period $\omega$ then all solutions
of (\ref{pce}) with initial values $x_{-1}=s$ and $x_{0}=t_{0}se^{-s}$ (i.e.
$(x_{-1},x_{0})$ is on the curve $g_{0}$) have period $p\omega.$

(b) If the map $f$ has a non-periodic point, then (\ref{pce}) has a
non-periodic solution.

(c) If $f$ has a period-three point then (\ref{pce}) has periodic solutions of
period $pn$ for all positive integers $n$ as well as chaotic solutions in the
sense of Li-Yorke \cite{LY}.
\end{corollary}

\begin{proof}
By Lemma \ref{pet} $\{t_{n}\}$ has period $p$ so the application of Lemma
\ref{pq} completes the proof.
\end{proof}

\medskip

If $p=2$ in the above corollary then $a_{0}=a_{1}$ (because $\sigma=0$) so the
sequence $\{a_{n}\}$ is constant, i.e. it has mimimal period 1 not 2. The
behavior described in the corollary is indeed that which is observed for the
constant parameter case; see \cite{LS} for a discussion of the stability of
the variety of solutions mentioned above, which is of the same type as noted
in Remark \ref{ms}.

\section{Summary and future directions}

We used a semiconjugate factorization of (\ref{pce}) to investigate its
dynamics. Semiconjugate factorizations for difference equations of exponential
type are not generally known (unlike linear equations) but fortunately we have
one for (\ref{pce}). As we see above, the decomposition of (\ref{pce}) into
the system (\ref{star1})-(\ref{star2}) of first-order equations makes it clear
why the solutions of (\ref{pce}) behave differently in a fundamental way
depending on whether the period of $\{a_{n}\}$ is odd or even: in the former
case the sequence $\{t_{n}\}$ is periodic, hence bounded while the latter case
$\{t_{n}\}$ is unbounded when $\sigma\not =0.$

The main results of this paper are Corollary \ref{op} and Theorem \ref{oddgen}
that discuss the dynamics of solutions if $p$ is odd and Theorem \ref{peven}
if $p$ is even. Corollary \ref{op} and Theorem \ref{oddgen} show that
(\ref{pce}) has multistable coexisting solutions, including non-periodic and
chaotic solutions if the amplitude of the parameter sequence $a_{n}$ is
unrestricted. Theorem \ref{peven} indicates a completely different dynamics
where globally stable limit cycles occur when $a_{n}$ is restricted to the
interval (0,2). Another of our main results is Theorem \ref{per} that extends
previous special cases in \cite{FHL}\ and \cite{LS}. Further, Corollary
\ref{op} is a consequence of Theorem \ref{per}.

An extension of Theorem \ref{peven} that includes non-periodic solutions when
$a_{n}$ exceeds 2 for some indices $n$ is expected and may be of future
interest. Such an extension evidently proves the existence of asymptotically
stable non-periodic solutions, including chaotic solutions for (\ref{pce})
when $\sigma\not =0$.

A natural extension of the above results is not obvious for higher order
versions of (\ref{pce}) such as
\begin{equation}
x_{n+1}=x_{n-1}e^{a_{n}-x_{n}-x_{n-k}} \label{kth}%
\end{equation}

For instance, (\ref{kth}) may have unbounded solutions if $k\geq2$ and exhibits
different dynamics than (\ref{pce}). Further, known semicongjugate
factorizations for (\ref{kth}) decompose it into a factor equation with order
at least 2 if $k\geq2$; see \cite{FSOR}. Such an equation is less tractable
than the first-order case studied above. A detailed study of difference
equations such as (\ref{kth}) and similar with periodic $\{a_{n}\}$ may yield
interesting and possibly unexpected results.

\end{document}